\documentclass[a4paper,11pt]{article}
\usepackage[latin1]{inputenc}
\usepackage[T1]{fontenc}
\usepackage[english]{babel}
\usepackage{eurosym}
\usepackage[pdftex,usenames,dvipsnames]{color}
\usepackage[pdftex]{graphicx}

\usepackage{hyperref}

\usepackage{caption2}
\usepackage{amsmath}
\usepackage{amssymb}
\usepackage{fancybox}
\usepackage{amscd}
\usepackage{amsthm}

\theoremstyle{plain}
\newtheorem{Thm}{Theorem}
\newtheorem{Cor}{Corollary}

\newtheorem{Prop}{Proposition}
\theoremstyle{definition}

\theoremstyle{remark}

\numberwithin{equation}{section}

\title{{\bf Indices of the iterates of \({\Bbb R}^3\)-homeomorphisms at fixed
points which are isolated invariant sets}}

\author{Patrice Le Calvez, Francisco R. Ruiz del Portal and José M. Salazar \thanks{ The authors have been supported by MEC,
MTM 2006-0825, the first author has also been supported by ANR (Symplexe, ANR-06-BLAN-0030-01)\newline 2000 {\em Mathematics Subject
Classification}: 37C25, 37B30, 54H25.
}}

\begin{document}

\maketitle

\begin{abstract}
Let \(U \subset {\mathbb R}^3\) be an open set and \(f:U
\rightarrow f(U) \subset {\mathbb R}^3\) be a homeomorphism.  Let
\(p \in U\) be a fixed point. It is known that, if  \(\{p\}\) is
not an isolated invariant set, the sequence of the fixed point
indices of the iterates of \(f\) at \(p\), \((i(f^n,p))_{n\geq
1}\), is, in general, unbounded. The main goal of this paper is to
show that when \(\{p\}\) is an isolated invariant set, the
sequence \((i(f^n,p))_{n\geq 1}\) is periodic. Conversely, we show
that for any periodic sequence of integers \((I_n)_{n \geq1}\) satisfying Dold's necessary congruences, there exists an
orientation preserving homeomorphism such that \(i(f^n,p)=I_n\) for every \(n\geq 1\). Finally we
also present an application to the study of the local structure of
the stable/unstable sets at \(p\).
\end{abstract}

\maketitle

\centerline{{\bf 1. Introduction.}}

\medskip

\medskip

Let \(U \subset {\mathbb R}^m\) be an open set and \(f:U
\rightarrow f(U) \subset {\mathbb R}^m\) be a continuous map. Let
\(p \in U\) be a fixed point of $f$ that is an isolated fixed
point of $f^n$ for every $n\geq 1$.  Since Dold \cite{Do}, it is
known that the sequence \((i(f^n, p))_{n\geq 1}\) of Lefschetz
indices must satisfy some rules, called Dold's congruences. Shub
and Sullivan proved that for \(C^1\)-maps the sequence is bounded.
Chow, Mallet-Paret and Yorke \cite{CMPY} gave bounds about the
form of the sequence of indices in terms of the spectrum of the
derivative \(Df(p)\). Babenko and Bogatyi \cite{BB} proved that
these bounds are sharp in dimension 2 and in a more recent paper
\cite{GN} Graff and Nowak-Przygodzki have proved  that for \(n=3\)
and \(C^1\)-maps, the sequence of fixed point indices follows one
among exactly seven different periodic patterns.

Suppose now that $f$ is a homeomorphism. For \(m=2\), see
\cite{LY}, \cite{LY1},  \cite{L} and \cite{RS},  it is known that
the sequence of indices \((i(f^n, p))_{n \geq 1}\) is periodic
with a very particular behavior. This sequence contains important
dynamical information. For $m\geq 3$,  the behavior is completely
different from that of diffeomorphisms. If an isolated fixed point
\(p\) is not isolated as invariant set, the sequence of indices of
the iterates in general is not bounded, even if the fixed point is
Lyapunov stable. More precisely, in \cite{RS1} it is proved that
any sequence satisfying Dold's congruences is realized as the
sequence of fixed point indices of the iterates of a \({\mathbb
R}^m\)-homeomorphism at an isolated and stable fixed point (see
also \cite{RS2}).

What happens if  \(p\) is an isolated invariant set ? There are
examples, for $m=8$, where the sequence \((i(f^n, p))_{n\geq 1}\)
is unbounded (see \cite{P}). The main result of the paper asserts
that such a situation cannot occur if $m=3$, more precisely:

\begin{Thm}
Let $U \subset {\mathbb R}^3$ be an open set and \(f:U \subset
{\mathbb R}^3 \rightarrow f(U) \subset {\mathbb R}^3\) be a
homeomorphism. Let $p$ be a fixed point of $f$ such that \(\{p\}\)
is an isolated invariant set. Then the sequence
\((i(f^n,p))_{n\geq 1}\) is periodic.
\end{Thm}

From   \cite{BB} (see  also \cite{JM}), one knows that a bounded
sequence that satisfies Dold's congruences is periodic. Therefore
it would be sufficient to prove that the sequence is bounded in the main
theorem. In fact we will prove directly the periodicity. If $f:
M\to M$ is a continuous map on a $m$-dimensional compact manifold
$M$, denote by $f_{*i}$ the induced map on the singular homology
group $H_i(M,{\mathbb Z})$. A classical result of Manning \cite{M}
asserts that ${\rm h}_{\rm top}(f)\geq \log({\rm sr}(f_{*1}))$,
where ${\rm h}_{\rm top}(f)$ is the topological entropy of $f$ and
${\rm sr}(f_{*1})$  the spectral radius of $f_{*1}$. As explained
by Manning, this implies by Poincar\'e duality that if $f$ is a
homeomorphism of a $3$-dimensional compact manifold, one has ${\rm
h}_{\rm top}(f)\geq\sup_{0\leq i\leq 3} \log({\rm sr}(f_{*i}))$.
The inequality ${\rm h}_{\rm top}(f)\geq\sup_{0\leq i\leq n}
\log({\rm sr}(f_i^*))$ is  known to be true if $f$ is $C^{\infty}$
(see \cite {Y}) but fails to be true is $f$ is only continuous~:
what is constructed in \cite{P} is an example of a homeomorphism
on a $8$-dimensional manifold where the previous inequality does
not occur. As we will see later Manning's result will be the
key-point of the proof of the main result. We will need to apply
it, not on a manifold but on an absolute neighborhood retract ANR
constructed via a filtration pair and we will need to use
Lefschetz-Dold's formula to get a duality result. In fact we will
prove the following more precise result.

\begin{Thm}
Let $U \subset {\mathbb R}^3$ be an open set and \(f:U \subset
{\mathbb R}^3 \rightarrow f(U) \subset {\mathbb R}^3\) be a
homeomorphism. Let $S$ be a connected compact isolated invariant
set such that $h_{\rm top}(f\vert_S)=0$. Then the sequence
\((i(f^n,S))_{n\geq 1}\) is periodic.
\end{Thm}

We will denote by ${\rm cl}(Y)$, ${\rm int} (Y)$ and $\partial Y$
the closure, the interior and the frontier of a subset $Y$ of a
topological space $X$.

\pagebreak

{\bf 1.1. Lefschetz index}

\medskip

The reader is referred to \cite{Br2}, \cite{Do} and \cite{Nu} for
information about the Lefschetz index theory. Let us denote by
$S^{m-1}$ the unit sphere of the euclidian space ${\mathbb R}^m$.
Let $U\subset {\mathbb R}^m$ be a neighborhood of $0$ and $f:U\to
f(U)$ be a continuous map having an isolated fixed point at $0$.
If $\varepsilon>0$ is small enough, the map
$$z\mapsto\frac{ \varepsilon z-f(\varepsilon z)}{\Vert \varepsilon z-f(\varepsilon z)\Vert}$$
is well defined on  $S^{m-1}$ and its degree does not depend on
$\varepsilon$, it is the {\it Lefschetz index} $i(f,0)$. For
example,  if $f$ is differentiable at $0$ and $1$ is not an
eigenvalue of $Df(0)$, this number is equal to $(-1)^l$, where $l$
is the number of real eigenvalues larger than $1$.  Recall the
following facts in the case where $f$ is a homeomorphism~:

\begin{itemize}

\item[-] if $f$ preserves the orientation, then $i(f^{-1},0)=i(f,0)$ if $m$ is even and $i(f^{-1},0)=-i(f,0)$ if $m$ is odd;

\item[-]  if $f$ reverses the orientation, then $i(f^{-1},0)=-i(f,0)$ if $m$ is even and $i(f^{-1},0)=i(f,0)$ if $m$ is odd.

\end{itemize}
Using charts one can define similarly the Lefschetz index of a
continuous map on a manifold at an isolated fixed point $x$. Such
a definition has been extended by Dold to the case where $f$ is
defined on an ANR.  Among the properties of the index, just recall
that if $x$ is an attracting fixed point (i.e. there exists a
compact neighborhood $V$ of $x$ satisfying $f(V)\subset V$ and
$\bigcap_{n\geq 0}f^n(V)=\{x\}$) then one has $i(f,x)=1$. If
$S\subset X$ is a compact invariant set (i.e. $f(S)=S$) and if
there is a neighborhood $V$ of $S$ whose fixed points are all
contained in $S$, one may define the index $i(f,S)\in {\mathbb
Z}$. In the case where there are finitely many fixed points in
$S$, it is equal to the sum of all Lefschetz indices of fixed
points that are in $S$. A particular case is the case where $X$ is
compact and $S=X$. The number $i(f,X)$ is called the {\it Lefschetz
number} and denotes by $\Lambda(f)$. It is related to the action
on the singular homology groups by the Lefschetz-Dold formula
$$\Lambda(f)=\sum_{i=0}^n (-1)^{i}{\rm tr}(f_{*i}),$$
where $f_{*i}$ is the morphism induced by $f$ on $H_i(X,{\mathbb
Z})$ and $n$ the biggest integer such that $H_i(X,{\mathbb
Z})\not=0$.

\medskip

\medskip

{\bf 1.2. Discrete Conley index, filtration pairs}

\medskip
Let \(U \subset M\) be an open set of a $m$-dimensional manifold
$M$ and \(f:U \rightarrow f(U) \subset M$ be a homeomorphism. A
compact invariant set \(S\) is {\em isolated with respect to}
\(f\) if there exists a compact neighborhood \(N\) of \(S\) such
that $\bigcap_{k\in{\mathbb Z}}f^{-k}(N)=S\). The neighborhood
\(N\) is called an {\em isolating neighborhood of \(S\)}. Let us
say that $S$ is an {\em attractor} (resp. a {\em repeller}) if
there exists a compact neighborhood \(N\) of \(S\) such that
$\bigcap_{k\in{\mathbb N}}f^{k}(N)=S$ (resp.
$\bigcap_{k\in{\mathbb N}}f^{-k}(N)=S$).  Attractors and repellers
are special classes of isolated invariant sets. An {\em isolating
block} \(N\) is a compact set such that \({\rm cl}({\rm int}(N))=N\)
and \(f^{-1}(N) \cap N \cap f(N) \subset {\rm int}(N)\). Isolating
blocks are a special class of  isolating neighborhoods.

We consider the {\em exit set of}  a set \(N\subset U\) to be
defined as
\[
N^{-}=\{x \in N: f(x) \notin {\rm int}(N)\}.
\]

Let \(S\) be an isolated invariant set and suppose \(L \subset N\)
is a compact pair contained in $U$. The pair \((N,L)\) is called a
{\em filtration pair} for \(S\) (see Franks and Richeson paper
\cite{FR})  provided \(N\) and \(L\) are each the closure of their
interiors and

\begin{itemize}

\item[1)] \({\rm cl}(N \setminus L)\) is an isolating neighborhood of \(S\),

\item[2)] \(L\) is a neighborhood of \(N^{-}\) in \(N\),

\item[3)] \(f(L) \cap {\rm cl}(N \setminus L) = \emptyset\).

\end{itemize}

\medskip
We will recall in the next section how to construct a filtration
pair $(N,L)$ that is simplicial. In that case, the quotient space
$N_L$ obtained from the pair $(N,L)$ by an identification of $L$
to a point $[L]$ is an ANR.  The properties of filtration pairs
imply that this identification induces a continuous map $\bar f
:N_L \rightarrow N_L$ that fixes $[L]$ and that sends each point
$x\not\in L$ to the projection in $N_L$ of $f(x)$. The fixed point
$[L]$ is an attractor in the following strong sense~: the map
$\bar f$ is locally constant equal to $[L]$ on a neighborhood of
$[L]$. In fact the dynamics of $\bar f$ is easy to understand :
for every point $x$, either there exists an integer $n\geq 0$ such
that $\bar f^n(x)=[L]$ or its $\omega$-limit set is included in
$S$. One can apply Lefschetz-Dold formula to each iterate of $\bar
f$. Using the fact that $[L]$ is an attracting fixed point  and
that $H_i(N_L,{\mathbb Z})\sim H_i(N,L,{\mathbb Z})$, one gets
\[
\Lambda(\bar f^n)=\sum_{i=0}^m(-1)^{i}{\rm tr}(f_{*i}),
\]
where
\[
\Lambda(\bar f^n)=i(\bar f^n, [L])+i(\bar f^n, S)=1+i( f^n, S)
\]

The second aim of this article is to prove the following result:

\begin{Thm}
For any periodic sequence of integers \((I_n)_{n\geq 1}\) satisfying Dold's necessary congruences, there exists an
orientation preserving homeomorphism of ${\mathbb R}^3$ that fixes $\{0\}$, that does not fixed any other compact set, such that
\(\{0\}\) admits an isolating block which is a topological ball and such that
 \(i(f^n,0)=I_n\) for every $n\geq 1$.
\end{Thm}

The paper is organized as follows: the next section is devoted to the proof
of the main theorem of this paper and to present its dynamical consequences.
Section 3 is dedicated to Theorem 3.

\medskip
\medskip

\centerline{{\bf 2. Proof of the Main Theorem.}}

\medskip
\medskip

{\bf 2.1. Topological entropy}

\medskip

\medskip

The topological entropy is a numerical invariant related to the
orbits growth. It represents the exponential growth rate for the
number of orbit segments distinguishable with arbitrary fine
precision. Let $f:X \rightarrow X$ be a continuous map with $X$ a
compact metric space. We define an increasing sequence of metrics
$d_n^f$, $n=1,2,\dots,$ with $d_1^f=d$,
\[
d_n^f(x,y)=\max_{0 \leq i \leq n-1}d(f^i(x),f^i(y)).
\]

We denote the open balls
\[
B_f(x,\epsilon,n)=\{y \in X : d_n^f(x,y)< \epsilon\}.
\]

A set $E \subset X$ is said to be $(n, \epsilon)$-{\em spanning}
if $X \subset \bigcup_{x \in E} B_f(x, \epsilon, n)$. If
$S_d(f,\epsilon,n)$ is the minimal cardinality of an
$(n,\epsilon)$-spanning set then we define the topological entropy
of $f$ as
\[
h_{\rm top}(f)=\lim_{\epsilon \to 0} \limsup_{n \to \infty}
\frac{1}{n} \log S_d(f,\epsilon,n).
\]

This definition does not depend on the metric $d$ but only on the
topology of $X$.

It is known that the restriction of  a map to the non-wandering
set $\Omega(f)$ captures the entropy of the system. If $\Omega(f)$
is the (closed and invariant) set of points $x \in X$ such that
for every open neighborhood $U$ of $x$ there is $n>0$ such that
$f^n(U) \cap U \neq \emptyset$, then we have the following
equality
\[
h_{\rm top}(f)=h_{\rm top}(f|_{\Omega(f)})
\]
 due to Bowen (see \cite{KH}, \cite{Ro} or
\cite{W} for details).

Recall now the following result of Manning (Theorem 2 in \cite{M})
which can be applied to spaces with {\em nice} local properties
like finite CW-complexes of compact ANR's.

\begin{Thm}
If $X$ is a compact metric space with metric \(d\) such that:

i) For every \(\epsilon >0\) there exists \(\delta >0\) such that
if \(d(x,y) < \delta \) then there exists a path \(\sigma: [0,1]
\rightarrow X\) from \(x\) to \(y\) with \({\rm diam}
(\sigma([0,1]) < \epsilon\).

ii) There exists \(\epsilon_0 >0\) such that any loop of diameter
\(< \epsilon_0\) is homotopically trivial in \(X\).

Then for any continuous map $f:X \rightarrow X$
\[
h_{\rm top}(f) \geq \log(|\lambda|)
\]
\noindent for every eigenvalue $\lambda$ of $f_{*1}:H_1(X,{\mathbb
Z}) \rightarrow H_1(X,{\mathbb Z})$.
\end{Thm}

\medskip
\medskip
\noindent{\bf 2.2. Some algebraic lemmas}

\medskip
Recall the classical following results~:

\begin{Prop}
Let $A$ and $B$ two finitely generated ${\mathbb C}$-vector spaces
and $u:A\to A$ and $v: B\to B$ two linear maps. If ${\rm
tr}(u^k)={\rm tr}(v^k)$ for every $k\geq 1$, then $u$ and $v$ have
the same non-zero eigenvalues counted with their multiplicities
\end{Prop}

\begin{proof}

If $\{\lambda_1, \dots, \lambda_r\}$ and $\{\mu_1, \dots, \mu_s\}$
are the non-zero eigenvalues of $u$ and $v$ respectively, we have
\[
\sum_{k=1}^{\infty}\frac{{\rm
tr}(u^k)}{k}Z^k=\sum_{k=1}^{\infty}\frac{\lambda_1^k+\cdots+\lambda_r^k}{k}Z^k=-\sum_{i=1}^r
\log (1-\lambda_i Z).
\]
The hypothesis gives us the polynomial equality
\[
\prod_{i=1}^r (1-\lambda_i Z)= \prod_{j=1}^s (1-\mu_j Z),
\]
\noindent which implies that $r=s$ and that for every \(j\) there
exists \(i_j\) such that $\mu_j=\lambda_{i_j}$.

\end{proof}

\begin{Prop}
Let $A$ be a finitely generated $\mathbb Z$-module and $u:A\to A$
be a morphism. If all the eigenvalues of $u$ have a modulus $\leq
1$, then all the non-zero eigenvalues are roots of unity.
\end{Prop}

\begin{proof}
The characteristic polynomial of $u$ may be written
$$P_u(X)=X^s\prod_{1\leq i\leq r} (X-\lambda_i),$$ where $\lambda_1, $\dots, $\lambda_r$ are the non-zero eigenvalues of $u$.
The coefficients of $P_u$ being integers, one deduces that
$\prod_{1\leq i\leq r}\lambda_i\in{\mathbb Z}$.  By hypothesis,
this implies that each $\lambda_i$ satisfies
$\vert\lambda_i\vert=1$ and may be written
$\lambda_i=e^{2i\pi\alpha_i}$, where $\alpha_i\in{\mathbb
R}/{\mathbb Z}$. One knows that every point of  ${\mathbb
R}^r/{\mathbb Z}^r$ is a recurrent point of the rotation
$$(t_1,\dots,t_r)\mapsto (t_1+\alpha_1,\dots, t_r+\alpha_r).$$
One deduces that there exists $n>0$ such that every complex number
$\lambda_i^n=e^{2i\pi n\alpha_i}$ is close to $1$. This implies
that ${\rm tr}(u^n)=\lambda_1^n+\dots+ \lambda_r^n$ is close to
$r$. But this number being an integer must be equal to $r$ and
this implies that each $\lambda_i^n$ is equal to $1$.
\end{proof}

\medskip
\medskip

{\bf 2.3. Proof of Theorem 2. }

\medskip
\medskip
Let us begin with a more detailed construction of filtration pairs
(see also \cite{FR}).

\begin{Prop}
Let \(U \subset {\mathbb R}^m\) be an open set and \(f:U
\rightarrow f(U) \subset {\mathbb R}^m\) be a homeomorphism.  If
$S$ is a connected invariant compact set isolated with respect to
$f$, there exists a filtration pair \((N, L)\) for \(S\) that is
homeomorphic to a finite simplicial pair: the sets $N$, $L$ and
${\rm cl}(N \setminus L)$ are compact topological $m$-dimensional
manifolds. More precisely:

\begin{itemize}

\item[1)] if $S$ is neither an attractor, nor a repeller, one may suppose that $N$ is connected, that $L$ is not empty and
that no bounded component of ${\mathbb R}^m\setminus L$ is
included in $N$;

\item[2)] if $S$ is a repeller, one may suppose that $N$ is connected, that $L$ is not empty and
that there exists a unique bounded component of ${\mathbb
R}^m\setminus L$ that is included in $N$ and this component
contains $S$;

\item[3)] if $S$ is an attractor, one may suppose that $N$ is connected and that $L$ is empty.

\end{itemize}

\end{Prop}

\begin{proof}
Suppose first that $S$ is neither an attractor, nor a repeller.
One constructs first an isolating block, that means  an isolating
neighborhood  \(N\) of  \(S\) satisfying
$$f(N)\cap N\cap f^{-1} (N)\subset{\rm int}(N),$$ (see \cite{FR}) for example).
Replacing $N$ by a small neighborhood that is a $n$-dimensional manifold, one may suppose that $N$ is itself a manifold. The connected component $N'$ of $N$ that contains $S$ satisfies a similar equality because $\partial N'\subset \partial N$.
Replacing $N$ by $N'$ one may suppose that $N$ is connected.
The exit set $N^-$ is not empty because $S$ is not an attractor. The inclusions $N^-\subset N$ and $f(N^-)\cap N\subset \partial N$ imply that
$$f(N^-)\cap N\cap f^{-1}(N) =\emptyset.$$

One can find a small neighborhood $L$ of $N^-$ in $N$ satisfying
$$f(L)\cap N\cap f^{-1}(N) =\emptyset ,$$
 such that $(N,L)$ is a finite simplicial pair. Observe that $(N,L)$ is a filtration pair because ${\rm cl}(N\setminus L)\subset N\cap f^{-1}(N)$.

Write $W_i$, $1\leq i\leq r$ for the bounded connected components
of ${\mathbb R}^n\setminus L$ that are included in $N$ and define
$L'=L\bigcup_{1\leq i\leq r} W_i$. To get the proposition, one
must prove that $(N,L')$ is a filtration pair. As one knows that
$$ {\rm cl}(N\setminus L')\subset {\rm cl}(N\setminus L)\subset N\cap f^{-1}(N),$$
it is sufficient to prove that for every $i\in\{1,\dots, r\}$ one
has $S\cap {\rm cl}(W_i)=\emptyset$ and $ f(W_i)\cap N\cap
f^{-1}(N) =\emptyset$. From $$f(\partial W_i)\cap {\rm cl}
(W_i)\subset f(L)\cap {\rm cl}(N\setminus L)=\emptyset,$$ one
deduces that either ${\rm cl }(W_i)\subset f(W_i)$ or ${\rm cl}
(W_i)\cap{\rm cl}(f(W_i))=\emptyset$. In the first case the
connected $S$ must be included in $W_i$ and equal to
$\bigcap_{k\geq 0} f^{-k}({\rm cl} (W_i))$. It would be a
repeller. So one knows that   ${\rm cl} (W_i)\cap{\rm
cl}(f(W_i))=\emptyset$ and therefore that $S \cap {\rm cl}
(W_i)=\emptyset$.

The inclusion $W_i\subset N\setminus N^-$ implies that
$f(W_i)\subset {\rm int}(N)$ and therefore that $f(\partial
W_i)\subset N$. The inclusion $\partial W_i\subset L$ implies that
 $$f(\partial W_i)\cap f^{-1}(N)\subset f(L)\cap N\cap f^{-1}(N)=\emptyset.$$ By connectedness  of $f^{-1}(N)$, one deduces
  that either $f^{-1}(N)\subset f(W_i)$ or $f^{-1}(N)\cap{\rm cl}(f(W_i))=\emptyset$.
  The first case does not occur because $S\subset f^{-1}(N)$ and $S\cap  f(W_i)=\emptyset$.
  So we are in the second case and we deduce that   $ f(W_i)\cap N\cap f^{-1}(N) =\emptyset$.

 \medskip
 In the case where $S$ is a repeller, one may find an isolating block $N$ such that $N\subset f({\rm int}(N))$.
 We can suppose that it is a connected  $m$-dimensional and construct in a similar way a pair $(N,L)$.
 Observe that $\partial N\subset N^-\subset L$. To get the proposition,
 keeping the same notations as before, we must prove that one of the $W_i$
contains $S$.
 In other words we must prove that the connected component of ${\mathbb R}^m\setminus L$ that contains $S$ is bounded and contained in $N$.
  It is a consequence of the fact that $W_i$ meets $N$ (it contains $S$) but does not meet $\partial N$.

In the case where $S$ is an attractor one may find an isolating
block $n$ such that $f(N)\subset{\rm int}(N)$. Here again one may
suppose that $N$ is a connected compact $m$-dimensional manifold.
The exit set is empty and $(N,\emptyset)$ satisfies the
proposition.
\end{proof}

Suppose that $(N,L)$ is a filtration pair that satisfies the
previous proposition and write $\bar f: N_L\to N_L$ for the
induced map.

\begin{Prop}
In the cases 1) and 2) one has  $H_0(N_L,{\mathbb Z})\sim\mathbb
Z$,  in the case 3) one has  $H_0(N_L,{\mathbb Z})\sim{\mathbb
Z}^2$, in all cases one has $\bar f_{*0}={\rm Id}$.

 In the cases 1) and 3) one has  $H_m(N_L,{\mathbb Z})=\{0\}$,  in the case 2) one has $H_m(N_L,{\mathbb Z})\sim\mathbb Z$.
 In this last case one has $\bar f_{*m}={\rm Id}$ if $f$ preserves the orientation and  $\bar f_{*m}=-{\rm Id}$ if $f$ reverses the
 orientation.

\end{Prop}

\begin{proof}

In cases 1) and 2), the set $N$ being connected and $L$ being non
empty, one knows that $N_L$ is connected. Therefore
$H_0(N_L,{\mathbb Z})={\mathbb Z}$ and obviously $\bar f_{*0}={\rm
Id}$. In case 3), $N_L$ is no more connected because one add an
isolated point. This point is fixed and the other component $N$ is
sent in itself, therefore $H_0(N_L,{\mathbb Z})\sim{\mathbb Z}^2$
and $\bar f_{*0}={\rm Id}$.

\medskip
Recall that $H_m(N_L,{\mathbb Z})$ is equal to the relative
homology group $H_m(N,L,{\mathbb Z})$. By Alexander's duality
theorem (see \cite{Sw}, page 313)  one knows that
$$H_m(N,L,{\mathbb Z})\sim H^0({\mathbb R}^m\setminus L,\,{\mathbb
R}^m\setminus N,\,{\mathbb Z}).$$ Therefore one must look at
locally constant function defined on ${\mathbb R}^m \setminus L$
that vanishes on ${\mathbb R}^m \setminus N$.  Equivalently, one
must look at bounded connected components of ${\mathbb R}^m
\setminus L$ that are included in $ N$. There is none in cases 1)
and 3) , there is exactly one in case 2).  In this case a
generator of $H_m(N,L,{\mathbb Z})$ is given by $(\overline W,
\partial W)$ where $W$ is the connected component of ${\mathbb
R}^m\setminus L$ that contains $S$.  Observe now that  the pair
$(f(\overline W), f(\partial W))\subset (N,L)$ is homologous to
$(\overline W,
\partial W)$ if $f$ preserves the orientation and to -$(\overline
W, \partial W)$ if $f$ reverses the orientation.

\end{proof}

Let us prove now Theorem 2. We suppose that $m=3$ and that $h_{\rm
top}( f\vert_S)=0$. We choose a filtration pair
 $(N,L)$ satisfying Proposition 3. Suppose first that $S$ is neither an attractor, nor a repeller. From the two equalities
\[
\Lambda(\bar f^n)={\rm tr}(\bar f^n_{*_0})- {\rm tr}(\bar
f^n_{*1})+{\rm tr}(\bar f^n_{*2})-{\rm tr}(\bar f^n_{*3})
\]
and
\[
\Lambda(\bar f^n)=i(\bar f^n, [L])+i(\bar f^n, S)=1+i( f^n, S),
\]
and from Proposition 4 we deduce  that
\[
i( f^n, S)=- {\rm tr}(\bar f^n_{*1})+{\rm tr}(\bar f^n_{*2}).
\]

To prove that the sequence $(i( f^n, S)_{n\geq 1}$ is bounded, it
is sufficient to prove that the spectral radius of both maps $\bar
f^n_{*1}$ and $\bar f^n_{*2}$ are bounded by $1$. Then,
Proposition 2 will imply that the eigenvalues are roots of unity,
which will give us the periodicity condition. The fact that the
spectral radius of the first map is bounded by $1$ is a
consequence of Manning's theorem. Indeed, the set of non wandering
points $\Omega(\bar f)$ is  included in $\{[L]\}\cup S$ because
$\Omega(\bar f)\setminus\{[L]\}$ is a closed set included in $N
\setminus L$ and invariant by $f$. Therefore
\[
h_{\rm top}(\bar f)=h_{\rm top}(\bar f|_{\Omega(\bar f)})=h_{\rm
top}(f|_{S})=0.
\]
 To prove a similar result for $\bar f_{*,2}$ we will use the following duality argument.  Let
$(N,L')$ be a filtration pair for $f^{-1}$ satisfying Proposition
3 and let us consider the quotient space $N_{L'}$ and the induced
map $\overline{f^{-1}}:N_{L'} \rightarrow N_{L'}$.

The pointed union $M=N_L \vee N_{L'}$ of $N_L$ and $N_{L'}$
obtained by the identification of the points $[L]$ and $[L']$ to a
point $*$ is a compact metric ANR and $*$ is an attractor of the
induced map $\bar g=\bar f \vee \overline{ f^{-1}}: M \rightarrow
M$. Applying Lefschetz-Dold formula, one gets

\[
\Lambda(\bar g^n)={\rm tr}(\bar g^n_{*_0})- {\rm tr}(\bar
g^n_{*1})+{\rm tr}(\bar g^n_{*2})-{\rm tr}(\bar g^n_{*3})
\]
where

\[
\Lambda(\bar g^n)=i(\bar g^n, *)+i(f^n, S)+i(f^{-n},S)\]

We shall assume first that \(f\) is orientation preserving. Then,
$i(f^n,S)=-i(f^{-n},S)$ and it follows that $\Lambda(\bar g^n)=1$.
On the other hand, one knows that ${\rm tr}(\bar g^n_{*_0})=1$,
that each space $H_i(M,{\mathbb Z})$, $1\leq i \leq 3$ may be
written $H_i(M,{\mathbb Z})= H_i(N_L, {\mathbb Z)}\oplus
H_i(N'_{L'}, {\mathbb Z)}$ and that each map $\bar g_{*i}$ may be
decomposed as $\bar g_{*i}= \bar f_{*i}\oplus \overline
{f^{-1}}_{*i}$. One deduces that
\[
-{\rm tr}(\bar g_{*1}^n)+{\rm tr}(\bar g_{*2}^n)=0,
\]
for every $n\geq 0$. By Proposition 1,  this implies that the non
vanishing eigenvalues of $\bar g_{*1}$ are equal to the non
vanishing eigenvalues of $\bar g_{*2}$, with the same
multiplicities. The spectral radius of  $\bar g_{*1}$ being $\leq
1$ because $h_{\rm top}(\bar g)=0$, we have a similar result for
$\bar g_{*2}$ and therefore a similar result for $\bar f_{*2}$.

In the case where $f$ is orientation reversing, we will get
\[
-{\rm tr}(\bar g_{*1}^n)+{\rm tr}(\bar g_{*2}^n)=0,
\]
for every even number $n$. But this will be sufficient to prove
that the non vanishing eigenvalues of $\bar g_{*1}$ are equal, up
to the sign, to the non vanishing eigenvalues of $\bar g_{*2}$.

\medskip

Is $S$ is an attractor of $f$, then it is a repeller of $f^{-1}$.
We construct filtration pairs $(N,L)$ and $(N', L')$ for $f$ and
$f^{-1}$ that satisfy Proposition 3. If $f$ is orientation
preserving, by looking at $\Lambda(\bar f)$ and $\Lambda(\overline
{f^{-1}})$ we get
\begin{eqnarray*}
1+i(f^n, S) & = & 2- {\rm tr}(\bar f^n_{*1})+{\rm tr}(\bar f^n_{*2}), \\
1+i(f^{-n}, S) & = & 1- {\rm tr}((\overline{f^{-1}}_{*1})^n)+{{\rm
tr}((\overline{f^{-1}}_{*2})^n)-1.}
\end{eqnarray*}
and
\[
1=\Lambda(\bar g^n)=2- {\rm tr}(\bar g^n_{*1})+{\rm tr}(\bar
g^n_{*2})-1
\]
and we deduce that
\[
-{\rm tr}((\bar g_{*1}^n)+{\rm tr}((\bar g_{*2}^n)=0.
\]
We can conclude as in the first case. If $f$ is orientation
reversing  we will get the equality for even integers and this
permits us to conclude. The case where $S$ is a repeller is
treated in the same way.

\medskip
Observe that we have proven the following:

\begin{Cor}
Let $U \subset {\mathbb R}^3$ be an open set and \(f:U \subset
{\mathbb R}^3 \rightarrow f(U) \subset {\mathbb R}^3\) be a
homeomorphism. Let $p$ be a fixed point of $f$ such that \(\{p\}\)
is an isolated invariant set.

\begin{itemize}

\item[-]  If $p$ is an attracting fixed point, the sequence \((i(f^n,p))_{n\geq 1}\) is constant equal to $1$;

\item[-]  If $p$ is a repulsing fixed point and $f$ preserves the orientation, the sequence \((i(f^n,p))_{n\geq 1}\) is constant equal to $-1$;

\item[-]  If $p$ is a repulsing fixed point and $f$ reverses the orientation, one as $i(f^n,p)=(-1)^{n+1}$;

\item[-]  if $p$ is neither an attracting fixed point, nor a repulsing fixed point, there exist roots of unity $\lambda_1,\dots,\lambda_r,\mu_1,\dots, \mu_s$ such that
$$i(f^n,p)=-\sum_{i=1}^r \lambda_i^n+\sum_{j=1}^s\mu_j^n.$$

\end{itemize}
\end{Cor}

\medskip
\medskip

{\bf Remark 1}. Let $U \subset {\mathbb R}^3$ be an open set,
\(f:U \subset {\mathbb R}^3 \rightarrow f(U) \subset {\mathbb
R}^3\) be an orientation preserving homeomorphism and \(K \subset
U\) be a continuum that is an isolated invariant set. Using the
same ideas as above we have that there exist
$\lambda_1,\dots,\lambda_r,\mu_1,\dots, \mu_s$ such that
$$i(f^n,K)=-\sum_{i=1}^r \lambda_i^n+\sum_{j=1}^s\mu_j^n.$$
Taking into account that
\(h_{\text{top}}(f|_K)=h_{\text{top}}(f^{-1}|_K)\) we have that

\[
\limsup_n \frac{\log |i(f^n,K)|}{n} \leq h_{\text{top}}(f|_K).
\]

\medskip
\medskip

{\bf 2.4. Local Stable/unstable sets.}

\medskip

Let $U \subset {\mathbb R}^3$ be an open set and let \(f:U \subset
{\mathbb R}^3 \rightarrow f(U) \subset {\mathbb R}^3\) an
orientation preserving homeomorphism. Let \(p\) be a fixed point
that  is an isolated invariant set and that is neither an
attractor nor a repeller. Let us consider a filtration pair
\((N,L)\) and let us define the {\em
local unstable set of} \(p\) {\em in} \({\rm cl} (N \setminus
L)\):
$$\Lambda_{{\rm cl} (N \setminus L)}^{-}(p)=\bigcap_{n\geq 0} f^{n}({\rm  cl}(N \setminus L)).$$
A point $x$ belong to $\Lambda_{{\rm cl} (N \setminus L)}^{-}(p)$
if and only if the sequence $(f^{-n}(x))_{n\geq 0}$ is well
defined and takes its values in  \({\rm cl} (N \setminus L)\).
Obviously the set $ \Lambda_{{\rm cl} (N \setminus L)}^{-}(p)$ is
backward invariant, which implies that $\bigcap_{n\geq 0} f^{-n}(
\medskip\Lambda_{{\rm cl} (N \setminus L)}^{-}(p))$ is invariant.
As ${\rm cl} (N \setminus L)$ is an isolating neighborhood of
$\{p\}$, the set $\bigcap_{n\geq 0} f^{-n}( \medskip\Lambda_{{\rm
cl} (N \setminus L)}^{-}(p))$ is reduced to the singleton $\{p\}$.
Therefore, every $x\in\Lambda_{{\rm cl} (N \setminus L)}^{-}(p)$
satisfies
$$\lim_{n\to\infty} f^{-n}(x)=p.$$

\medskip
We will prove the following:

\begin{Thm} In the situation of the above paragraph, suppose in addition that \(i(f^q,p)=r> 0 \) , where $q$  is the period of the sequence
\((i(f^n,p))_{n\geq1}\).

\begin{itemize}

\item[a)]  The second
Alexander-$\Check{C}$ech cohomology group with compact supports
\(\Check{H}^2_c(\Lambda_{{\rm cl}(N \setminus L)}^{-}(p) \setminus
L,{\mathbb Z})\) contains at least \(r\) copies of \({\mathbb
Z}\).

\item[b)] If \ \(N_L\) can be embedded in \({\mathbb
R}^3\) then the image of \((\Lambda_{{\rm cl}(N \setminus
L)}^{-}(p) \cup L)/L\) by this embedding decomposes \({\mathbb
R}^3\) into at least \(r+1\) components.

\item[c)] The set \(\Lambda_{{\rm cl}(N \setminus
L)}^{-}(p) \setminus L\) decomposes \( {\rm int}(N) \setminus L\)
into at least \(r+1 - {\rm dim} \,H_1({\rm cl}(N \setminus L),
{\mathbb Z})\) components. In particular, if \({\rm cl}(N
\setminus L)\) is a closed ball, \(\Lambda_{{\rm cl}(N \setminus
L)}^{-}(p) \setminus L\) decomposes \({\rm int}(N) \setminus L\)
into at least \(r+1\) components.

\end{itemize}

\end{Thm}

\begin{proof}

The proof will involve shape theory arguments. More precisely, the
homology and cohomology theory that is a shape invariant. It is
more appropriate, to study spaces that can have bad local
behavior, than the singular homology theory. We recommend the book
of Marde\(\Check{\text{s}}\)i\(\Acute{\text{c}}\) and Segal
(\cite{MS}) for information about the theory of shape.

We keep the same notations as in the previous sections. The set
$${\rm Inv}^{-}(N_L, \bar{f})=\bigcap_{n \geq 0}
\bar{f}^{n}(N_L)$$is a continuum, being the intersection of a
decreasing sequence of continua (recall that $N$ is connected and
$L$ not empty). Observe that it is the image of $\Lambda_{{\rm
cl}(N \setminus L)}^{-}(p)$ by the projection $\pi:N\to N_L$.
Indeed if $x\in\Lambda_{{\rm cl}(N \setminus L)}^{-}(p)$ then for
every $n\geq 0$, the point $x$ has a pre-image $f^{-n}(x)$ which is
in $N$ and one has $\bar f^n(\pi(f^{-n}(x)))=\pi(x)$. Conversely,
suppose that $\xi \in ({\rm Inv}^{-}(N_L, \bar{f})$  is not equal
to $[L]$ and equal to the projection of $x\in N\setminus L$. A
point $\eta$ such that $\bar f^n(\eta)=\xi$ is not equal to $[L]$.
Such a point is unique, it must be the projection of $f^{-n}(x)$.
We have proven that
$${\rm Inv}^{-}(N_L, \bar{f})=[L]\cup\pi\left ( \Lambda_{{\rm cl}(N \setminus L)}^{-}(p)\right).$$
As we know that the set on the left is connected and both sets on
the right are compact, we deduce that $[L]\in\pi\left (
\Lambda_{{\rm cl}(N \setminus L)}^{-}(p)\right)$. In other words,
${\rm Inv}^{-}(N_L, \bar{f})$ is the Alexandroff's
compactification of \(\Lambda_{{\rm cl}(N\setminus L)}^{-}(p)
\setminus L\) .

\medskip
 Without loss of generality we
can assume that \(q=1\). In this case we have that all the
eigenvalues of $\bar f_{*1}$ and $\bar f_{*2}$ are equal to 1.
Since the Lefschetz number can be computed also in terms of the
cohomology groups the same statement holds for
$$\bar f^{*}_1:H^1(N_L,{\mathbb Z})\to H^1(N_L,{\mathbb Z})$$ and
$$\bar f^{*}_2:H^2(N_L,{\mathbb Z})\to H^2(N_L,{\mathbb Z}).$$

The set  \({\rm Inv} ^{-}(N_L, \bar{f})\) is the inverse limit of
the sequence
\[
\dots N_L \overset{\bar{f}}{\rightarrow} N_L
\overset{\bar{f}}{\rightarrow} N_L \overset{\bar{f}}{\rightarrow}
N_L \dots
\]
Since \(i(f,p)=r > 0 \), there are at least \(r\) eigenvalues of
$\bar f^{*}_2$ that are equal to 1. As a consequence, the
(co)pro-group associated to the above inverse system
\[
\dots H^2(N_L,{\mathbb Z}) \overset{\bar f_2^*}{\rightarrow}
H^2(N_L,{\mathbb Z}) \overset{\bar f_2^*}{\rightarrow}
H^2(N_L,{\mathbb Z}) \overset{\bar f_2^*}{\rightarrow}
H^2(N_L,{\mathbb Z}) \dots
\]
is nontrivial and the Alexander-\(\Check{C}\)ech cohomology groups
\[
\Check{H}^2({\rm Inv}^{-}(N_L, \bar{f}),{\mathbb Z})
=\Check{H}^2((\Lambda_{{\rm cl}(N \setminus L)}^{-}(p) \cup
L)/L,{\mathbb Z})= \Check{H}^2(\Lambda_{{\rm cl}(N \setminus
L)}^{-}(p) \cup L,L,{\mathbb Z})
\]
\noindent contain at least \(r\) copies of \({\mathbb Z}\).

As we know that
\[
\Check{H}^2(\Lambda_{{\rm cl}(N \setminus L)}^{-}(p) \cup
L,L,{\mathbb Z})= \Check{H}^2_c(\Lambda_{{\rm cl}(N \setminus
L)}^{-}(p) \setminus L,{\mathbb Z}).
\]
(see \cite{Sp}, page 321, Lemma 11) we deduce that
\(\Check{H}^2_c(\Lambda_{{\rm cl}(N \setminus L)}^{-}(p) \setminus
L,{\mathbb Z})\) contains at least \(r\) copies of \({\mathbb
Z}\). This proves a).

In order to check b) let us consider an embedding $\iota:
N_L\to{\mathbb R}^3$. We know that
 $$\tilde{H}_0({\mathbb R}^3 \setminus \iota((\Lambda_{{\rm cl}(N \setminus
L)}^{-}(p) \cup L)/L), {\mathbb Z}) =
\Check{H}^2(\iota((\Lambda_{{\rm cl}(N \setminus L)}^{-}(p) \cup
L)/L),{\mathbb Z})$$ (see \cite{Sp}, page 296, Theorem 10). The
fact that  \(\Check{H}^2((\Lambda_{{\rm cl}(N \setminus L)}^{-}(p)
\cup L)/L),{\mathbb Z})\) contains \(r\) copies of \(\mathbb Z\)
implies that
 \(\iota((\Lambda_{cl(N \setminus L)}^{-}(p) \cup L)/L)\) decomposes
\({\mathbb R}^3\) into at least \(r+1\) components.

Now we are going to prove c). Using Theorem 10 in page 342 of
\cite{Sp}, we have that
\[
H_1({\rm int}(N) \setminus L, ({\rm int}(N) \setminus L) \setminus
(\Lambda_{{\rm cl}(N \setminus L)}^{-}(p) \setminus L), {\mathbb
Z}) = \Check{H}^2_c((\Lambda_{{\rm cl}(N \setminus L)}^{-}(p)
\setminus L),{\mathbb Z}).\]

Then, from a) and the exactness of the homology sequence of the
pair \(({\rm int}(N) \setminus L, ({\rm int}(N) \setminus L)
\setminus (\Lambda_{{\rm cl}(N \setminus L)}^{-}(p) \setminus
L))\), if follows that
\[
{\rm dim}\, H_0( {\rm int}(N) \setminus L) \setminus
(\Lambda_{{\rm cl}(N \setminus L)}^{-}(p) \setminus L), {\mathbb
Z}) \geq r+1 -{\rm dim}\,H_1({\rm cl}(N \setminus L), \mathbb Z).
\]

\end{proof}

\medskip
\medskip

{\bf Remark 2.} Note that using shape theory arguments it is easy
to check that the covering dimension \({\rm dim}( {\rm
Inv}^{-}(N_L, \bar{f})) \geq 2\). Indeed, \({\rm pro}-H_2({\rm
Inv}^{-}(N_L, \bar{f}),\mathbb Z)\) can be also computed in terms
of the \(\Check{C}\)ech expansion of \({\rm Inv}^{-}(N_L,
\bar{f})\), it follows that there is a cofinal family of open
coverings of \({\rm Inv}^{-}(N_L, \bar{f})\)  whose nerves are
polyhedra of dimension \(\geq2\). Then, there is a cofinal family
of open coverings of \({\rm Inv}^{-}(N_L, \bar{f})\)  of order
\(\geq 3\). Then, the covering dimension \({\rm dim}( {\rm
Inv}^{-}(N_L, \bar{f})) \geq 2\).

\medskip
\medskip

{\bf Remark 3.} If in the above theorem \(f\) is orientation
preserving and \(i(f^q,p)=r< 0 \) we obtain a similar result for
the local unstable set of \(\{p\}\) associated to \(f^{-1}\).

\medskip
\medskip

\centerline{{\bf 3. Proof of Theorem 3.}}

\medskip
\medskip

We have proved that the sequence $(i(f^n,p))_{n \geq 1}$  is periodic. Let us answer the
converse question, i.e., for any periodic sequence of integers
$I=(I_n)_{n \geq 1}$, which satisfies Dold's  necessary congruences, there
exists a homeomorphism with $i(f^n,p)=I_n$ for every $n\geq 1$. Moreover, we can
construct the homeomorphism $f:{\Bbb R}^3 \rightarrow {\Bbb R}^3$
to be  orientation preserving with no compact invariant sets but
\(\{p\}\).

There are different ways to define Dold's congruences. The simplest one uses the  normalized periodic sequences $\sigma^k=(\sigma^k_n)_{n\geq 1}$, where for each $k\geq 1$, one has
\[ \sigma^k_n=\left\{\begin{array}{ll}
k & \text{ if } n \in k{\Bbb N}, \\
0 & \text{ if } n \notin k{\Bbb N} \end{array} \right.
\]
Any sequence $I=(I_n)_{n \geq 1}$ may be written formally $I=\sum_{k\geq 1} a_k\sigma^k$ in an unique way.
It satisfies Dold's congruences if and only if the $a_k$ are integers. As proved by Dold (\cite {Do}),
a sequence of integers $I=(I_n)_{n \geq 1}$ satisfies  Dold's congruences
if and only if there exists an ENR $X$, a continuous map $f:X \rightarrow X$ and an open subset
$U \subset X$ such that for every $n\geq 1$, the set ${\rm Fix}(f^n) \cap U_n$ is
compact and $I_n=i_{X}(f^n, U_n)$, the sequence $(U_n)_{n\geq 1}$ being inductively defined by  $U_1=U$ and
$U_n=f^{-1}(U_{n-1}) \cap U$. As observed by  Babenko and Bogatyi
(see \cite{BB} or \cite{JM}), a bounded sequence that satisfies Dold's congruences must be periodic.
That means that the sum  $I=\sum_{k\geq 1} a_k\sigma^k$ is finite.

\medskip
To prove Theorem 3 we will find a map in the class of {\it radial}
homeomorphisms (see \cite{P}, § 8) . The $3$-dimensional sphere is
nothing but the end compactification of $S^2\times\mathbb R$ where
one adds to $S^2\times\mathbb R$ the lower end $e^-$, adherent to
$S^2\times(-\infty,0]$, and the upper end $e^+$, adherent to
$S^2\times[0,+\infty)$. Let us fix a homeomorphism $g$ of $S^2$
and a continuous map $\varphi:S^2\to\mathbb R$. The skew product
$$f_{g,\varphi}:(x,y)\mapsto (g(x), y+\varphi(x))$$ induces on $S^3$ a homeomorphism fixing the two ends.
If one wants $\{e^-\}$ to be locally maximal it is sufficient to suppose that the following property (P) is satisfied:

\smallskip
 (P)\enskip\enskip for every $x\in S^2$, $\varphi(x)\geq 0\Rightarrow \varphi(g(x))>0$.
 \smallskip

\noindent  Indeed the property (P) implies that there exists $\varepsilon >0$ such that for every $x\in S^2$, $\varphi(x)\geq -\varepsilon\Rightarrow \varphi(g(x))\geq \varepsilon$. Writing
$$f_{g,\varphi}^{-1}:(x,y)\mapsto (g^{-1}(x), y-\varphi(g^{-1}(x))),$$ one deduces  that  every orbit $(x_k,y_k)_{k\in{\Bbb Z}}$ of $f_{g,\varphi}$ satisfies at least one of the following properties:

\smallskip
\noindent-\enskip $y_{k+1}-y_k\geq\varepsilon$ for every $k\geq 0$

\smallskip
\noindent or

\smallskip
\noindent-\enskip $y_{k-1}-y_k\geq \varepsilon$ for every $k\leq 0$.

\smallskip
\noindent This implies that either $\lim_{k\to +\infty} y_k=+\infty$ or $\lim_{k\to -\infty} y_k=+\infty$.
This also implies that each ball $(S^2\times(-\infty,r])\cup\{e^-\}$ is an isolating block.

\medskip

Write $X_0=\varphi^{-1}([0,+\infty[)$. The following points are
easy to state:

\smallskip
\noindent-\enskip the exit set of
$N=(S^2\times(-\infty,0])\cup\{e^-\}$ is $N^-=\{(x,y)\,\vert\,x\in
X_0,\,-\varphi(x)\leq y\leq 0\}$;

\smallskip
\noindent-\enskip the unstable set of $e^-$ is $X^-\times{\Bbb R}$ where $X^-=\bigcap_{k\geq 0} g^k(X_0)$, it is the infinite cone over $X^-$;

\smallskip
\noindent-\enskip the stable set of $e^-$ is $X^+\times{\Bbb R}$ where $X^+=\bigcap_{k\leq 0} g^k(S^2\setminus X_0)$;

\smallskip
\noindent-\enskip one obtains a filtration pair $(N,L)$ writing
$L=\{(x,y)\,\vert\,x\in X,\,-\psi(x)\leq y\leq 0\}$ where
$X\subset S$ is a subsurface with boundary, neighborhood of $X_0$
satisfying $g(X)\subset {\rm int }(X_0)$, and
$\psi:X\to[0,\infty)$ a continuous function vanishing exactly on
$\partial X$ and satisfying $\psi(g(x))>\psi(x)-\varphi(x)>0$ for
every $x\in X$.

\medskip
One can compute the Lefschetz index $i(f_{g,\varphi}^n, e^-)$ by
looking at the induced maps on the spaces $H_i(N,L,{\Bbb Z})$. We
will give an alternative way, supposing for convenience that $X_0$
itself is a subsurface with boundary.

\begin{Prop}
For every $n\geq 1$, one has $i(f_{g,\varphi}^n, e^-)=1-\Lambda((g\vert_{X_0})^n)$, where
$$\Lambda((g\vert_{X_0})^n)={\rm tr} ((g\vert_{X_0})_{*0})^n- {\rm tr} ((g\vert_{X_0})_{*1})^n$$ is the Lefschetz number of the $n$-th iterate of $g\vert_{X_0}:X_0\to X_0$.
\end{Prop}

\begin{proof} Write $N^+=\{(x,y)\,\vert\,x\in X_0, \,0\leq y\leq \varphi(x)\}$ and observe that $N'=N\cup N^+$ is a $3$-dimensional sphere.
One constructs a continuous map $\overline f_{g,\varphi}: N'\to N'$ satisfying
\[ \overline f_{g,\varphi}(x)=\left\{\begin{array}{ll}
(g(x),0) & \text{ if } x\in N^-\cup N^+, \\
f_{g,\varphi}(x)  & \text{ if } x \notin N^-\cup N^+.\end{array} \right.
\]
Observe that both sets $N'$ and $N^-\cup N^+$ are attracting sets of $\overline f_{g,\varphi}$ and that $r:(x,y)\mapsto (x,0)$  is a strong deformation retraction from $N^-\cup N^+$ to $X_0$ that satisfies $r\circ\overline f_{g,\varphi}= g\vert_{X_0}\circ r$. This implies that
\begin{eqnarray*}
1 & = & \Lambda((\overline f_{g,\varphi})^n) \\
 & = & i(f_{g,\varphi}^n, e^-)+ \Lambda((\overline f_{g,\varphi}\vert_{N\cup N'})^n)\\
& =&i(f_{g,\varphi}^n, e^-)+ \Lambda((g\vert_{X_0})^n).
\end{eqnarray*}
 \end{proof}

 \medskip
Now we can prove Theorem 3.
 Let $I=(1+a_1)\sigma^1+\sum_{2\leq k\leq k_0}a_k \sigma^k$ be a periodic sequence that satisfies Dold's congruences.
 Write $$A^-=\{k\geq 1\,\vert\, a_k<0\},\enskip A^+=\{k\geq 1\,\vert \,a_k>0\},$$ and define
 $$a^-=\sum_{k\in A^-}a_k, \,a^+=\sum_{k\in A^+}a_k.$$Choose $-a^-$ closed disks $(Z^k_i)_{k\in A^-,\,1\leq i\leq-a_k}$ and $a^+$ closed
 annulus $(Z^k_i)_{k\in A^+,\,1\leq i\leq a_k}$ on $S^2$, all pairwise disjoint. One can easily construct an orientation preserving
 homeomorphism $g$ of $S^2$ that satisfies
 the following properties :

  \smallskip
 \noindent-\enskip for each $k\in A^-$ and each $i\in\{1,\dots,-a_k\}$ there exist $k$ disjoint closed disks  $D^k_{i,j}$, $j\in\mathbb Z/k\mathbb Z$,
 in the interior of $Z^k_i$ such that $g(D^k_{i,j})\subset{\rm int}(D^k_{i,j+1})$ and such that the maximal invariant set contained
 in $X^k_i=\bigcup_{j\in\mathbb Z/k\mathbb Z}D^k_{i,j}$ is a periodic orbit of period $k$;

 \smallskip
  \noindent-\enskip for each  $k\in A^+$  and each $i\in\{1,\dots,a_k\}$ there exist $k$ disjoint closed disks  $D^k_{i,j}$,
  $j\in\mathbb Z/k\mathbb Z$, in the interior of $Z^k_i$ such that $g^{-1}(D^k_{i,j})\subset{\rm int}(D^k_{i,j+1})$, such that
  $X^k_i=Z^k_i\setminus\bigcup_{j\in\mathbb Z/(k)\mathbb Z}{\rm int}(D^k_{i,j})$ is an attracting set of $g$ and such that the maximal
  invariant set contained in $X^k_i$ is a wedge of $k+1$ circles.

Choose now a continuous function $\varphi: S^2\to \mathbb R$ that is positive in the interior of
$X_0=  \bigcup_{1\leq k\leq k_0,\,1\leq i\leq\vert a_k\vert}X^k_i$ and negative outside $X_0$. Observe that
\[ \Lambda((g\vert_{X_i^k})^n)=\left\{\begin{array}{ll}
k & \text{ if } n \in k{\Bbb N}, \\
0 & \text{ if } n \notin k{\Bbb N} \end{array} \right.
\]
if $k\in A^-$, and
\[ \Lambda((g\vert_{X_i^k})^n)=\left\{\begin{array}{ll}
-k & \text{ if } n \in k{\Bbb N}, \\
0 & \text{ if } n \notin k{\Bbb N} \end{array} \right.
\]
if $k\in A^+$. By Proposition 5, one deduces that the sequence
$J=(i(f_{g,\varphi}^n, e^-))_{n\geq 1}$ satisfies
 \begin{eqnarray*}
J&=&\sigma^1-\sum_{k\in A^-}\sum_{1\leq i\leq -a_k} \sigma^k-\sum_{k\in A^+}\sum_{1\leq i\leq a_k} (-\sigma^k) \\
 &=&\sigma^1+\sum_{k\in A^-}a_k\sigma^k+\sum_{k\in A^+}a_k\sigma^k\\
 &=&I
\end{eqnarray*}
\(\square\).
\medskip
\medskip

{\bf Remark 4.} \enskip Of course, a given sequence $I$
corresponds to infinitely many such constructions. In the
construction above, the sequence $I^0=\sigma^1$ corresponds to a
map $f_{g,\varphi}$ such that $\varphi<0$ (the ball $N$ being an
attracting set), the sequence $I^1=-\sigma^1$ to a map
$f_{g,\varphi}$ such that $X_0$ is the union of two attracting
disks (an explicit example being given by the hyperbolic linear
map $l:(x_1,x_2,x_3)\mapsto (x_1/2, x_2/2, 2 x_3)$), the sequence
$I^2=\sigma^1-\sigma^2$ to a map $f_{g,\varphi}$ such that $X_0$
is the union of two permuting disks (an explicit example being
given by $l':(x_1,x_2,x_3)\mapsto (x_1/2,x_1/2, -2 x_3)$). Observe
that $I^0$ also corresponds to a map $f_{g,\varphi}$ such that
$X_0$ is an attracting annulus (an explicit example being given by
$l^{-1}$)  and that $I^1$ also corresponds to a map
$f_{g,\varphi}$ such that $\varphi>0$ (the ball $N$ being a
repulsing set).

\medskip
{\bf Remark 5.} \enskip Applying the previous proposition to
$f_{g,\varphi}^{-1}:(x,y)\mapsto (g^{-1}(x),
y-\varphi(g^{-1}(x)))$, one knows that $i(f_{g,\varphi}^{-n},
e^-)=1-\Lambda((g^{-1}\vert_{S^2\setminus g(X_0})^n)$ for every
$n\geq 1$. As we are on a $2$-dimensional sphere and as we know
that $g(X_0)\subset X_0$ we deduce that
$$\Lambda((g^{-1}\vert_{S^2\setminus g(X_0)})^n)=i(g^{-n},
X^+)=i(g^n, X^+),$$ which implies that
$$i(f_{g,\varphi}^{-n}, e^-)+(f_{g,\varphi}^n, e^-)= 2-i(g^n, X^-)-i(g^{n}, X^+)=0$$ by Poincar\'e-Hopf formula.

\medskip
{\bf Remark 6.} \enskip All the constructions above may be done in
higher dimension. Radial homeomorphisms may be constructed
similarly  and Proposition 5 may be generalized. Replacing all the
disks by balls and the annuli by filled tori in the construction
above, one gets a generalization of Theorem 3 in any dimension.
These constructions permit us to understand why Theorem 1 is not
true in higher dimension. Let consider the following Anosov
diffeomorphism
$$A:(x_1,x_2)\mapsto(2x_1+x_2,x_1+x_2)$$ on  ${\Bbb R}^2/{\Bbb Z}^2$ and the map
$$g_0: (x_1,x_2,x_3)\mapsto( A(x_1,x_2), x_3/2)$$defined on $({\Bbb R}^2/{\Bbb Z}^2)\times[-1,1]$.
Using Hopf's fibration, one easily constructs a homeomorphism $g$ on $S^3$ that admits an attracting manifold $M$ homeomorphic
to $({\Bbb R}^2/{\Bbb Z}^2)\times[-1,1]$ such hat $g\vert_M$ is conjugated to $g_0$. Let consider a function $\varphi: S^3\to{\Bbb R}$ positive
on the interior of $M$ and negative outside $M$. The unstable manifold of $e^-$ is an infinite cone over ${\Bbb R}^2/{\Bbb Z}^2$ (like in \cite{P})
and one knows that for every $n\geq 1$
$$i(f_{g,\varphi}^n, e^-)=1-L((g\vert_{M})^n)=1-L(A^n)=-1+\left(3+\sqrt 5\over 2\right)^n+\left(3-\sqrt 5\over 2\right)^n,$$
which implies that $\lim_{n\to +\infty}i(f_{g,\varphi}^n, e^-)=+\infty$.

\medskip
\medskip

Patrice Le Calvez

Institut de Math\'ematiques de Jussieu (UMR 7586 CNRS), UPMC 175
rue du Chevaleret 75013 Paris, France.

E-mail: lecalvez@math.jussieu.fr

\medskip
\medskip

Francisco R. Ruiz del Portal

Departamento de Geometr\'{\i}a y Topolog\'{\i}a, Facultad de
CC.Mate\-m\'{a}\-ti\-cas, Universidad Complutense de Madrid,
Madrid 28040, Spain.

E-mail: R\(_{-}\)Portal@mat.ucm.es

\medskip
\medskip

Jos\'e Manuel Salazar.

Departamento de Matem\'aticas. Universidad de Alcal\'a.  Alcal\'a
de Henares. Madrid 28871, Spain.

E-mail: josem.salazar@uah.es

\end{document}